\documentclass[10pt,draft,a4paper]{amsart}
\usepackage{amssymb}
\usepackage{graphicx}
\numberwithin{equation}{section}
\newtheorem{theorem}{Theorem}[section]

\newtheorem{proposition}[theorem]{Proposition}
\newtheorem{lemma}[theorem]{Lemma}
\theoremstyle{remark}
\newtheorem{remark}{Remark}[section]

\theoremstyle{definition}
\newtheorem{definition}{Definition}[section]

\newcommand{\R}{\mathbb{R}}

\newcommand{\Z}{\mathbb{Z}}

\begin{document}


\title[Magnetic Schr\"odinger equation]
{Endpoint Strichartz estimates for the magnetic Schr\"odinger
equation}

\author{Piero D'Ancona}
\address{Piero D'Ancona:
SAPIENZA - Universit\`a di Roma, Dipartimento di Matematica,
Piazzale A.~Moro 2, I-00185 Roma, Italy}
\email{dancona@mat.uniroma1.it}

\author{Luca Fanelli}
\address{Luca Fanelli:
Universidad del Pais Vasco, Departamento de
Matem$\acute{\text{a}}$ticas, Apartado 644, 48080, Bilbao, Spain}
\email{luca.fanelli@ehu.es}

\author{Luis Vega}
\address{Luis Vega: Universidad del Pais Vasco, Departamento de
Matem$\acute{\text{a}}$ticas, Apartado 644, 48080, Bilbao, Spain}
\email{luis.vega@ehu.es}

\author{Nicola Visciglia}
\address{Nicola Visciglia: Universit\'a di Pisa, Dipartimento di Matematica,
Largo B. Pontecorvo 5, 56100 Pisa, Italy}
\email{viscigli@dm.unipi.it}

\begin{abstract}
  We prove Strichartz estimates for the Schr\"odinger equation with
  an electromagnetic potential, in dimension $n\geq3$. The decay
  and regularity assumptions on the potentials are almost critical,
  i.e., close to the Coulomb case. In addition, we require
  repulsivity and a non trapping condition, which are expressed as
  smallness of suitable components of the potentials. However,
  the potentials themselves can be large, and we avoid completely any
  a priori spectral assumption on the operator.
  The proof is based on smoothing estimates and new Sobolev
  embeddings for spaces associated to magnetic potentials.
\end{abstract}

\date{\today}


\maketitle

\section{Introduction}\label{sec:introd}

Recent research on linear and nonlinear dispersive equations
is largely focused on measuring precisely the rate of decay of
solutions. Indeed, decay and Strichartz estimates are one of the
central tools of the theory, with immediate applications to
local and global well posedness, existence of low regularity
solutions, and scattering. This point of view includes most
fundamental equations of physics like
the Schr\"odinger, Klein-Gordon, wave and Dirac equations.
Strichartz estimates appeared in \cite{str}; the basic framework for this study was laid out in the two
papers \cite{GinibreVelo95-generstric}, \cite{KeelTao98-endpoinstric},
which examined in an exhaustive way
the case of constant coefficient, unperturbed equations.
This leads naturally to the possible extentions to equations
perturbed with electromagnetic potentials or with variable
coefficients; a general theory of dispersive properties for
such equations is still under construction and very actively
researched.

In the present paper we shall focus on the time dependent
Schr\"odinger equation
\begin{equation}\label{eq:schro}
    i\partial_tu(t,x)=Hu(t,x),\qquad
    u(0,x)=\varphi(x),\qquad
    x\in \mathbb{R}^{n},\quad n\ge3
\end{equation}
associated with the electromagnetic Schr\"odinger operator
\begin{equation}\label{eq:hamiltoniana}
  H:= -\nabla_A^2+V(x),
  \qquad
  \nabla_A:=\nabla-iA(x)
\end{equation}
where $A=(A^1,\dots,A^n):\R^n\to\R^n$, $V:\R^n\to\R$. We recall that
in the unperturbed case $A\equiv0$, $V\equiv0$, dispersive properties
are best expressed in terms of the mixed norms on $\mathbb{R}^{1+n}$
\begin{equation*}
  L^pL^q:= L^p(\R_t;L^{q}(\R^n_x))
\end{equation*}
as follows: for every $n\geq 3$,
\begin{equation*}
  \|e^{it\Delta}\varphi\|_{L^pL^q}\leq c_{n}\|\varphi\|_{L^2},
\end{equation*}
provided the couple $(p,q)$ satisfies the admissibility condition
\begin{equation}\label{eq:admis}
  \frac2p=\frac n2-\frac nq,
  \qquad
  2\leq p\leq\infty.
\end{equation}
These estimates are usually referred to as \emph{Strichartz estimates}.
Our main goal is to find sufficient conditions on the
potentials $A,V$ such that Strichartz estimates
are true for the perturbed equation \eqref{eq:schro}.

In the purely electric case $A\equiv0$ the literature
is extensive and almost complete; we may cite among many others
the papers \cite{BurqPlanchonStalker04-a}, \cite{GoldbergSchlag04-a},
\cite{RodnianskiSchlag04-a}. It is now clear that in
this case the decay
$V(x)\sim1/|x|^2$ is critical for the validity of Strichartz
estimates; suitable counterexamples were constructed in \cite{GVV}.
In the
magnetic case $A\neq0$, the Coulomb decay $|A|\sim1/|x|$ is likely to be
critical, however no explicit counterexamples are
available at the time. An intense research is ongoing concerning
Strichartz estimates for the magnetic
Schr\"odinger equation, see e.g.~\cite{pda-lf2},
\cite{ErdoganGoldbergSchlag-a},
\cite{ErdoganGoldbergSchlag-b}, \cite{GST}; see also
\cite{RobbianoZuily-a} for a more general class of first order
perturbations.

Due to the perturbative techniques used in the above
mentioned papers, an assumption concerning absence of zero-energy
resonances for the perturbed operator $H$ is typically required in order to
preserve the dispersion. In the case $A\equiv0$ it was shown in
\cite{BurqPlanchonStalker04-a} how this abstract condition can be
dispensed with, by directly proving some weak dispersive estimates
(also called Morawetz or smoothing estimates)
via multipliers methods. Here we shall
give a very short proof of Strichartz estimates
for the magnetic Schr\"odinger equation with potentials of almost
Coulomb decay, based uniquely on the weak dispersive estimates proved
in \cite{FV}. The leading theme is that direct multiplier techniques
allow to avoid, under suitable repulsivity
conditions on $V$ and non-trapping conditions on $A$ (see also
\cite{F}), the presence of non-dispersive components, and to preserve
Strichartz estimates.

We begin by introducing some notations.
Regarding as usual the potential $A$ as a 1-form, we define
the corresponding {\it magnetic field} as the 2-form $B=dA$,
which can be identified with the anti-symmetric gradient of $A$:
\begin{equation}\label{eq:B}
  B\in\mathcal M_{n\times n},
  \qquad
  B=DA-(DA)^t,
\end{equation}
where $(DA)_{ij}=\partial_iA^j$, $(DA)^t_{ij}=(DA)_{ji}$. In
dimension 3, $B$ is uniquely determined by the vector
field $\text{curl}A$ via the vector product
\begin{equation}\label{eq:B3}
  Bv=\text{curl}A\times v,
  \qquad
  \forall v\in\R^3.
\end{equation}
We define the {\it trapping component} of $B$ as
\begin{equation}\label{eq:Btau}
  B_\tau(x)=\frac{x}{|x|}B(x);
\end{equation}
when $n=3$ this reduces to
\begin{equation}\label{eq:Btau3}
  B_\tau(x)=\frac{x}{|x|}\times\text{curl}A(x),
  \qquad
  n=3,
\end{equation}
thus we see that $B_\tau$ is a tangential vector.
The trapping component may be intepreted as
an obstruction to the dispersion of solutions; some explicit
examples of potentials $A$ with $B_\tau=0$ in dimension 3
are given in \cite{F}, \cite{FV}.

Moreover, by
\begin{equation*}
  \partial_rV=\nabla V\cdot \frac x{|x|},\qquad
\end{equation*}
we denote the radial
derivative of $V$, and we decompose it into its positive and
negative part
\begin{equation*}
  \partial_rV=(\partial_rV)_+-(\partial_rV)_-.
\end{equation*}
The positive part $(\partial_rV)_+$ also represents an obstruction
to dispersion, and indeed we shall require it to be small in a
suitable sense. To ensure good spectral properties of the operator
we shall also assume that the negative part $V_{-}$ is not too large
in the sense of the Kato norm:

\begin{definition}\label{def:kato}
  Let $n\ge3$. A measurable function $V(x)$ is said to be in the
  \emph{Kato class} $K_{n}$ provided
  \begin{equation*}
    \lim_{r \downarrow0}\sup_{x\in \mathbb{R}^{n}}
    \int_{|x-y|\le r} \frac{|V(y)|}{|x-y|^{n-2}}dy=0.
  \end{equation*}
  We shall usually omit the reference to the space dimension and write
  simply $K$ instead of $K_{n}$.
  The \emph{Kato norm} is defined as
  \begin{equation*}
    \|V\|_{K}=\sup_{x\in \mathbb{R}^{n}}
    \int\frac{|V(y)|}{|x-y|^{n-2}}dy.
  \end{equation*}
\end{definition}
A last notation we shall need is the radial-tangential norm
\begin{equation*}
\|f\|_{L^p_rL^\infty(S_r)}^p:=\int_0^\infty \sup_{|x|=r}|f|^pdr.
\end{equation*}

In our results we always assume that the operators $H$ and
$-\Delta_A:=-(\nabla-iA)^2$ are self-adjoint and
positive on $L^2$, in order to ensure the existence of the
propagator $e^{itH}$ and of the powers $H^s$ via the
spectral theorem. There are several sufficient conditions
for selfadjointness and positivity, which can be expressed in terms
of the local integrability properties of the coefficients
(see the standard references \cite{CFKS}, \cite{LS});
here we prefer to leave this as an abstract assumption.
Our main result is the following:

\begin{theorem}\label{thm:main}
  Let $n\geq3$. Given $A,V\in C^1_{loc}(\R^n\setminus\{0\})$,
  assume the operators $\Delta_{A}=-(\nabla-iA)^{2}$ and
  $H=-\Delta_{A}+V$ are selfadjoint and positive on $L^{2}$.
  Moreover assume that
  \begin{equation}\label{eq:asskato}
    \|V_{-}\|_{K}<\frac{\pi^{\frac n2}}{\Gamma\left(\frac n2-1\right)}
  \end{equation}
  for a sufficiently small $\epsilon>0$ depending on $A$ and
  \begin{equation}\label{ass:decay}
    \sum_{j\in\Z}2^{j}\sup_{|x|\sim 2^{j}}|A|
    +
    \sum_{j\in\Z}2^{2j}\sup_{|x|\sim 2^{j}}|V|<\infty,
  \end{equation}
  and the Coulomb gauge condition
  \begin{equation}\label{eq:gauge}
   \mathrm{div}\ A=0.
  \end{equation}
  Finally, when $n=3$, we assume that for some $M>0$
  \begin{equation}\label{ass:smoo3D}
    \frac{\left(M+\frac12\right)^2}{M}\||x|^{\frac32}B_\tau\|_{L^2_rL^\infty(S_r)}^2
    +(2M+1)\||x|^2(\partial_rV)_+\|_{L^1_rL^\infty(S_r)}<\frac12,
  \end{equation}
  while for $n\geq4$ we assume that
  \begin{equation}\label{ass:smoo4D}
    \||x|^{2}B_\tau(x)\|_{L^{\infty}}^{2}
    +
    2\||x|^{3}(\partial_rV)_+(x)\|_{L^{\infty}}
    <\frac23(n-1)(n-3).
  \end{equation}
  Then, for any Schr\"odinger admissible couple $(p,q)$,
  the following Strichartz estimates hold:
  \begin{equation}\label{eq:strichartz}
    \|e^{itH}\varphi\|_{L^pL^q}\leq C \|\varphi \|_{L^2},
    \qquad
    \frac2p=\frac n2-\frac nq,
    \quad
    p\geq2,
    \quad
    p\neq2\ \text{if }n=3.
  \end{equation}
  In dimension $n=3$, we have the endpoint estimate
  \begin{equation}\label{eq:strichartz3D}
    \||D|^{\frac12}e^{itH}\varphi\|_{L^2L^{6}}\lesssim\|H^{\frac14}\varphi \|_{L^2},
  \end{equation}
\end{theorem}
\begin{remark}\label{rem:regularity}
  Let us remark that the regularity assumption 
  $A,V\in C^1_{\text{loc}}(\R^n\setminus\{0\})$ is actually 
  stronger than what we really require. 
  For the validity of the Theorem, we just need to give meaning 
  to inequalities \eqref{ass:smoo3D}, \eqref{ass:smoo4D}.
\end{remark}
\begin{remark}\label{rem:smoo}
  Assumptions \eqref{eq:gauge}, \eqref{ass:smoo3D}, and \eqref{ass:smoo4D}
  imply the weak dispersion of
  the propagator $e^{itH}$ (see Theorems 1.9, 1.10, assumptions (1.24), and (1.27) in
  \cite{FV}). Actually assumption (1.24) in \cite{FV} seems to be stronger than
  \eqref{ass:smoo3D}, but reading carefully the proof of Theorem 1.9 in \cite{FV}
  it is clear that the real assumption is our \eqref{ass:smoo3D} (see inequality (3.14) in \cite{FV}).
  The strict inequality in \eqref{ass:smoo3D},
  \eqref{ass:smoo4D} is essential, in order to dispose of the
  weighted $L^2$-estimate in the above mentioned Theorems by
  \cite{FV} (see also inequality \eqref{eq:fanellivega} below).
\end{remark}

\begin{remark}\label{rem:resonances}
  We emphasize that in Theorem \ref{thm:main} we do not require
  absence of resonances at energy zero, in contrast with
  \cite{ErdoganGoldbergSchlag-a}, \cite{ErdoganGoldbergSchlag-b}.
  Indeed, this is possible thanks to
  the non-trapping and repulsivity conditions
  \eqref{ass:smoo3D}, \eqref{ass:smoo4D};
  notice however that these conditions can be checked easily in
  concrete examples, which is not the case for the abstract assumption
  on resonances.
\end{remark}

The derivation of Strichartz estimates
from the weak dispersive ones turns out to be
remarkably simple if working on the half derivative
$|D|^{1/2}u$, see Section \ref{sec:main} for details.
As a drawback, the final estimates are expressed
in terms of fractional Sobolev spaces generated by the perturbed
magnetic operator $-\Delta_{A}$. Thus, in order to revert to standard
Strichartz norms as in \eqref{eq:strichartz}, we need suitable
bounds for the perturbed Sobolev norms in terms of the standard ones.
This is provided by the following theorem, which we think is
of independent interest.

\begin{theorem}\label{thm:equiv}
  Let $n\ge3$.
  Given $A\in L^{2}_{loc}(\R^n;\R^n)$, $V:\R^n\to\R$,
  assume the operators $\Delta_{A}=-(\nabla-iA)^{2}$ and
  $H=-\Delta_{A}+V$ are selfadjoint and positive on $L^{2}$.
  Moreover, assume that $V_{+}$ is of Kato class, $V_{-}$ satisfies
  \begin{equation}\label{eq:smallkato2}
    \|V_{-}\|_{K}<\frac{\pi^{\frac n2}}{\Gamma\left(\frac n2-1\right)},
  \end{equation}
  and
  \begin{equation}\label{eq:ass2A}
    |A|^{2}+V\in L^{n/2,\infty},\qquad
    A\in L^{n,\infty}.
  \end{equation}
  Then the following estimate holds:
  \begin{equation}\label{eq:magsob2A}
    \|H^{1/4}f\|_{L^{q}}\leq C_{q}\||D|^{\frac12}f\|_{L^q},\qquad
    1<q<2n,\qquad n\ge3.
  \end{equation}
  In addition we have the reverse estimate
  \begin{equation}\label{eq:magsob2B}
      \|H^{1/4}f\|_{L^{q}}\geq c_{q}\||D|^{\frac12}f\|_{L^q},\qquad
      \frac 43 <q <4,\qquad
      n\ge3.
  \end{equation}
\end{theorem}

\section{Proof of Theorem \ref{thm:equiv}}\label{sec:equiv}

We start with the proof of Theorem \ref{thm:equiv},
divided into several steps. First we
need to prove that the heat kernel associated with the operator $H$
is well behaved under quite general assumptions:
\begin{proposition}\label{pro:heatk}
  Consider the selfadjoint operator $H=-(\nabla-iA(x))^{2}+V(x)$
  on $L^{2}(\mathbb{R}^{n})$, $n\ge3$.
  Assume that $A\in L^{2}_{loc}(\R^n, \R^n)$, moreover the positive
  and negative parts $V_{\pm}$ of $V$ satisfy
  \begin{equation}\label{eq:Vpiu}
    V_{+}\ \text{is of Kato class},
  \end{equation}
  \begin{equation}\label{eq:Vmeno}
    \|V_{-}\|_K< c_{n}=\pi^{n/2}/\Gamma\left(n/2-1\right)).
  \end{equation}
  Then $e^{-tH}$ is an integral operator and its heat kernel
  $p_{t}(x,y)$ satisfies the pointwise estimate
  \begin{equation}\label{eq:heatest}
    |p_{t}(x,y)|\le
    \frac{(2\pi t)^{-n/2}}{1-\|V_{-}\|_{K}/c_{n}}
      e^{-|x-y|^{2}/(8t)}.
  \end{equation}
\end{proposition}
\begin{proof}
  We recall Simon's diamagnetic pontwise inequality
  (see e.g.~Theorem B.13.2 in \cite{simon}), which holds under weaker
  assumptions than ours: for any test function $g(x)$,
  \begin{equation*}
    |e^{t[(\nabla-iA(x))^{2}-V]}g|\le
    e^{t (\Delta-V)}|g|.
  \end{equation*}
  Notice that by choosing a delta sequence $g_{\epsilon}$ of (positive)
  test functions, this implies an analogous pointwise inequality
  for the corresponding heat kernels. Now we can apply the second
  part of Proposition 5.1 in \cite{pda-vp} which gives precisely
  estimate \eqref{eq:heatest} for the heat kernel of
  $e^{-t (\Delta-V)}$ under \eqref{eq:Vpiu}, \eqref{eq:Vmeno}.
\end{proof}
The second tool we shall use is a weak type estimate for imaginary
powers of selfadjoint operators, defined in the sense of spectral
theory. This follows easily from the previous heat kernel bound and
the techniques of Sikora and Wright (see \cite{sikora-wright}):
\begin{proposition}\label{pro:imag}
  Let $H$ be as in Proposition \ref{pro:heatk}, and assume in addition
  that $H\ge0$. Then for all
  $y\in \mathbb{R}$ the imaginary powers $H^{iy}$
  satisfy the (1,1) weak type estimate
  \begin{equation}\label{eq:weakest}
    \|H^{iy}\|_{L^{1}\to L^{1,\infty}}\le C(1+|y|)^{n/2}.
  \end{equation}
\end{proposition}
\begin{proof}
  By Theorem 3 in \cite{sikora} we obtain immediately that
  our heat kernel bound \eqref{eq:heatest} implies the finite
  speed of propagation for the wave kernel $\cos(t \sqrt{H})$,
  in the sense of \cite{sikora}, \cite{sikora-wright}, i.e.,
  \begin{equation*}
    (\cos(t \sqrt{H})\phi,\psi)_{L^{2}}=0
  \end{equation*}
  for all $\phi,\psi\in L^{2}$ with support in
  $B(\xi_{1},x_{1})$, $B(\xi_{2},x_{2})$ respectively, provided
  $|t|<2^{-1/2}(|x_{1}-x_{2}|-\xi_{1}-\xi_{2})$. Then we are
  in position to apply Theorem 2 from \cite{sikora-wright}
  which gives the required bound.
\end{proof}
We are ready to prove the first part of Theorem \ref{thm:equiv}.
\begin{proof}[Proof of \eqref{eq:magsob2A}]
  We shall use the Stein-Weiss interpolation theorem applied to
  the analytic family of operators
  \begin{equation*}
    T_{z}=H^{z}\cdot(-\Delta)^{-z}.
  \end{equation*}
  Here $H^{z}$ is defined by spectral theory while $(-\Delta)^{-z}$
  e.g. by the Fourier transform.
  Writing $z=x+iy$, we can decompose
  \begin{equation*}
    T_{z}=H^{iy}H^{x}(-\Delta)^{-x}(-\Delta)^{-iy},\qquad
    y\in \mathbb{R},\quad x\in[0,1].
  \end{equation*}
  The operators $H^{iy}$ and $(-\Delta)^{iy}$ are obviously bounded
  on $L^{2}$. On the side $\Re z=0$ the operator
  reduces to the composition of pure imaginary powers
  \begin{equation*}
    T_{iy}=H^{iy}(-\Delta)^{-iy}
  \end{equation*}
  and by the weak type estimate \eqref{eq:weakest} we obtain
  immediately by interpolation
  that $H^{iy}$, and hence
  $T_{iy}$ is bounded on $L^{p}$ for all $1<p<\infty$:
  \begin{equation}\label{eq:Re0}
    \|T_{z}f\|_{L^{p}}\le C(1+|y|)^{n/2}\|f\|_{L^{p}}
    \quad\text{for}\quad \Re z=0,\quad 1<p<\infty.
  \end{equation}
  Next we consider the case $\Re z= 1$.
  We start by proving the estimate
  \begin{equation}\label{eq:Dest}
    \|H f\|_{L^{r}}\leq C\|\Delta f\|_{L^{r}},\qquad
    1<r<\frac n2.
  \end{equation}
  For $f\in C^{\infty}_{c}(\R^n)$ we can write
  \begin{equation*}
    Hf=-\Delta f-2i A \cdot \nabla f+(|A|^{2}-i\nabla \cdot A+V)f.
  \end{equation*}
  We have then by H\"older's inequality in Lorentz spaces
  and assumption \eqref{eq:ass2A}
  \begin{equation*}
    \|A \cdot \nabla f\|_{L^{r}}\leq C
    \|A\|_{L^{n,\infty}}\|\nabla f\|_{L^{\frac{nr}{n-r},r}}\qquad
    1\le r<n,
  \end{equation*}
  and using the precised Sobolev embedding
  \begin{equation*}
    \|g\|_{L^{\frac{nr}{n-r},r}}\leq
    C\|\nabla g\|_{L^{r,r}}
    =\|\nabla g\|_{L^{r}}
  \end{equation*}
  (and the boundedness of Riesz operators, which rules out the
  case $r=1$) we obtain
  \begin{equation*}
    \|A \cdot \nabla f\|_{L^{r}}\leq C
    \|\Delta f\|_{L^{r}},\qquad 1<r<n.
  \end{equation*}
  In a similar way,
  \begin{equation*}
    \|(|A|^{2}-i\nabla \cdot A+V)f\|_{L^{r}}\leq C
    \||A|^{2}-i\nabla \cdot A+V\|_{L^{n/2,\infty}}
    \|f\|_{L^{\frac{nr}{n-2r},r}}\qquad
    1\le r<\frac n2
  \end{equation*}
  and again by the Sobolev embedding
  \begin{equation*}
    \|f\|_{L^{\frac{nr}{n-2r},r}}\leq C
    \|\Delta f\|_{L^{r}}
  \end{equation*}
  we conclude that
  \begin{equation*}
    \|(|A|^{2}-i\nabla \cdot A+V)f\|_{L^{r}}
    \leq C\|\Delta f\|_{L^{r}},\qquad
    1<r<\frac n2.
  \end{equation*}
  Summing up we obtain \eqref{eq:Dest}.
  Combining \eqref{eq:Dest} with the $L^r$-boundedness of the purely
  imaginary powers $H^{iy}$ and $-\Delta^{iy}$, we get
  \begin{equation}\label{eq:Tuno}
    T_{1+iy}:L^{r}\to L^{r},\qquad
    1<r<\frac n2.
  \end{equation}
  Interpolating \eqref{eq:Tuno} with
  \eqref{eq:Re0} we obtain
  \begin{equation*}
    \|T_{1/4}f\|_{L^{p}}\leq C\|f\|_{L^{p}}
  \end{equation*}
  for
  \begin{equation*}
    \frac1q=\frac 3{4p}+\frac1{4r}\qquad
    1<p<\infty,\quad
    1<r<\frac n2
    \implies\qquad
    1<q<2n
  \end{equation*}
  which concludes the proof.
\end{proof}

We pass now to the proof of the reverse estimate  \eqref{eq:magsob2B}. 
We shall need the following lemma:

\begin{lemma}\label{rem:selfadj}
  Assume that
  \begin{equation}\label{eq:katoass}
    A\in L^2_{loc}(\R^n),
    \qquad
    \|V_{-}\|_{K}< 4\pi^{n/2}/\Gamma(n/2-1).
  \end{equation}
  Then for some constnt $a<1$ the following inequality holds:
  \begin{equation}\label{eq:katoin}
    \int V_{-}|f|^{2}dx\le a\|\nabla_A f\|_{L^{2}}^{2}.
  \end{equation}
\end{lemma}

\begin{proof}
  The proof follows a standard
  argument. We begin by showing that
  \begin{equation}\label{eq:katoin00}
    \int V_-|f|^2dx\leq a\|\nabla f\|_{L^2}^2,
  \end{equation}
  for some $a<1$. This
  can be restated as
  \begin{equation*}
    (V_{-}^{1/2}(-\Delta)^{-1/2}f,V_{-}^{1/2}(-\Delta)^{-1/2}f)\le
      a\|f\|^{2},\qquad a<1
  \end{equation*}
  i.e. we must prove that the operator
  $T=V_{-}^{1/2}(-\Delta)^{-1/2}$ is bounded
  on $L^{2}$ with norm
  smaller than one. Equivalently, we must prove that the operator
  \begin{equation*}
    TT^{*}=V_{-}^{1/2}(-\Delta)^{-1}V_{-}^{1/2}
  \end{equation*}
  satisfies
  \begin{equation*}
    \|V_{-}^{1/2}(-\Delta)^{-1}V_{-}^{1/2}f\|^{2}\le
    b\|f\|^{2},\qquad b<1.
  \end{equation*}
  Writing explicitly the kernel of $(-\Delta)^{-1}$, we are reduced
  to prove
  \begin{equation*}
    I=\int|V_{-}(x)|\left|\int \frac{|V_{-}(y)|^{1/2}}{|x-y|^{n-2}}
      f(y)dy\right|^{2}dx\le k_{n}^{2}b\|f\|^{2}
  \end{equation*}
  where $k_{n}=4\pi^{n/2}/\Gamma(n/2-1)$, $b<1$. Now by Cauchy-Schwartz
  \begin{equation*}
    I\le
    \int|V_{-}(x)|\left(\int \frac{|V_{-}(y)|}{|x-y|^{n-2}}dy\right)
      \left(\int \frac{|f(y)|^{2}}{|x-y|^{n-2}}
        dy\right)dx
  \end{equation*}
  which gives
  \begin{equation*}
    I \le  \|V_{-}\|_{K}
      \iint \frac{|V_{-}(x)|}{|x-y|^{n-2}}|f(y)|^{2}dydx=
      \|V_{-}\|_{K}^{2}\|f\|^{2}
  \end{equation*}
  and this proves \eqref{eq:katoin00} under the smallness assumption 
  \eqref{eq:katoass}.
  Applying the same computation to the function $|f|$ instead of $f$,
  we deduce from \eqref{eq:katoin00} that
  \begin{equation}\label{eq:katoin000}
    \int V_-|f|^2dx\leq a\|\nabla|f|\,\|_{L^2}^2.
  \end{equation}
  Since $A\in L^2_{\text{loc}}$, we can apply the diamagnetic inequality
  \begin{equation*}
    |\nabla|f||\leq|\nabla_Af|,
    \qquad
    \text{a.e. in }\R^n
  \end{equation*}
  (see e.g. \cite{LL}) to obtain \eqref{eq:katoin}.
\end{proof}

\begin{proof}[Proof of \eqref{eq:magsob2B}]
  We begin by proving the $L^{2}$ inequality
  \begin{equation}\label{eq:L2ineq}
    \|(-\Delta)^{1/2}f\|_{L^{1}}\simeq
    \|\nabla f\|_{L^{2}}\le C\|H^{1/2}f\|_{L^{2}}.
  \end{equation}
  We can write, with the notation $\nabla_{A}=\nabla-iA$,
  \begin{equation*}
    \|H^{1/2}f\|^{2}=(Hf,f)=-(\nabla_{A}^{2}f,f)+\int V|f|^{2}=
       \|\nabla_{A}f\|^{2}_{L^{2}}+\int V|f|^{2}
  \end{equation*}
  and this implies
  \begin{equation*}
    \|H^{1/2}f\|^{2}\ge \|\nabla_{A}f\|^{2}_{L^{2}}-
      \int V_{-}|f|^{2}.
  \end{equation*}
  Thus by Lemma \ref{rem:selfadj} we have for some $a<1$
  \begin{equation*}
    \|H^{1/2}f\|^{2}\ge (1-a)\|\nabla_{A}f\|^{2}_{L^{2}}
  \end{equation*}
  so that, in order to prove \eqref{eq:L2ineq}, it is
  sufficient to prove the inequality
  \begin{equation}\label{eq:ineqb}
    \|\nabla f\|_{L^{2}}\le C
    \|\nabla_{A}f\|_{L^{2}}.
  \end{equation}
  Now,
  using as in the first half of the proof the H\"older inequality and the
  Sobolev embedding in Lorentz spaces, we can write
  \begin{equation*}
    \|A|f|\|^{2}_{L^{2}}+\int V|f|^{2}\le
    C\||A|^{2}+V\|_{L^{n/2,\infty}}\||f|\|_{L^{\frac{2n}{n-2},2}}
    \le  
    C\|\nabla|f|\|_{L^{2}}
  \end{equation*}
  by assumption \eqref{eq:ass2A}. Then, by the diamagnetic inequality
  \begin{equation*}
    |\nabla|f||\leq|\nabla_Af|
  \end{equation*}
  we obtain
  \begin{equation}\label{eq:L21}
    \|Af\|^{2}_{L^{2}}+\int V|f|^{2}\le C\|\nabla_{A}f\|^{2}_{L^{2}}.
  \end{equation}
  Moreover we have
  \begin{equation*}
    |(\nabla-iA)f|^{2}\ge\Bigl||\nabla f|-|Af|\Bigr|^{2}
    \implies
    |\nabla f|^{2}\le 2|\nabla_{A}f|^{2}+2|Af|^{2}
  \end{equation*}
  which implies
  \begin{equation*}
    \|\nabla f\|^{2}_{L^{2}}\le
    2\|\nabla_{A}f\|^{2}_{L^{2}}+2\|Af\|^{2}_{L^{2}}
  \end{equation*}
  and combinig this with \eqref{eq:L21} we get
  \begin{equation*}
    \|\nabla f\|^{2}_{L^{2}}\le
    C \|\nabla_{A} f\|^{2}_{L^{2}}-2\int V|f|^{2}\le
    C \|\nabla_{A} f\|^{2}_{L^{2}}-2\int V_{-}|f|^{2}.
  \end{equation*}
  Using again Lemma \ref{rem:selfadj} we finally arrive at 
  \eqref{eq:ineqb}, so that the claimed estimate
  \eqref{eq:L2ineq} is proved.
    
  Now we can use again the Stein-Weiss interpolation theorem, applied to
  the analytic family of operators
  \begin{equation*}
    T_{z}=(-\Delta)^{z}H^{-z}
  \end{equation*}
  with $z$ in the range $0\le \Re z\le 1/2$. 
  Writing $z=x+iy$ we have
  \begin{equation*}
    T_{z}=(-\Delta)^{iy}(-\Delta)^{x}H^{-x}H^{-iy},\qquad
    y\in \mathbb{R},\quad x\in[0,1/2].
  \end{equation*}
  The operators $H^{-iy}$ and $(-\Delta)^{iy}$ are bounded
  on $L^{2}$, while estimate \eqref{eq:L2ineq} proves that
  $(-\Delta)^{\frac 12}H^{-\frac 12}$ is also bounded on $L^{2}$.
  This shows that
  \begin{equation}\label{eq:Re12}
    \|T_{z}f\|_{L^{2}}\le C\|f\|_{L^{2}}
    \quad\text{for}\quad \Re z=1/2.
  \end{equation}
  On the side $\Re z=0$ the operator
  reduces to the composition of pure imaginary powers
  \begin{equation*}
    T_{iy}=(-\Delta)^{iy}H^{-iy}
  \end{equation*}
  and arguing as in the proof of \eqref{eq:Re0} we get that
  $T_{iy}$ is bounded on $L^{p}$ for all $1<p<\infty$:
  \begin{equation}\label{eq:Re02}
    \|T_{z}f\|_{L^{p}}\le C(1+|y|)^{n/2}\|f\|_{L^{p}}
    \quad\text{for}\quad \Re z=0,\quad 1<p<\infty.
  \end{equation}
  Then by the Stein-Weiss interpolation theorem we obtain as above
  \begin{equation*}
    \|T_{1/4}f\|_{L^{q}}\le C\|f\|_{L^{q}}
  \end{equation*}
  with
  \begin{equation*}
    \frac1q=\frac14+\frac{1}{2p},\qquad
    1<p<\infty \qquad
    \implies \qquad
    \frac43< q<4.
  \end{equation*}
\end{proof}

\section{Proof of Strichartz estimates}\label{sec:main}
Let us first recall some well known facts about the free propagator.
First of all, the free Strichartz estimates for $T(t)=e^{it\Delta}$,
its dual operator and the operator $TT^*$ are
\begin{equation}\label{eq:free1}
  \|e^{it\Delta}\varphi\|_{L^pL^q}\leq C\|\varphi\|_{L^2},
\end{equation}
\begin{equation}\label{eq:free2}
  \left\|\int e^{-is\Delta}F(s,\cdot)ds\right\|_{L^2}\leq C\|F\|_{L^{\widetilde p'}L^{\widetilde
  q'}},
\end{equation}
\begin{equation}\label{eq:free3}
  \left\|\int_0^t e^{i(t-s)\Delta}F(s,\cdot)ds\right\|_{L^pL^q}\leq C\|F\|_{L^{\widetilde p'}L^{\widetilde
  q'}},
\end{equation}
for all Schr\"odinger admissible couples $(p,q),(\widetilde
p,\widetilde q)$ satisfying
\begin{equation*}
  \frac 2p=\frac n2-\frac nq,
  \qquad
  p\geq2,
\end{equation*}
with $p\neq2$ if $n=2$ (see \cite{GinibreVelo95-generstric},
\cite{KeelTao98-endpoinstric}). Moreover, we recall the following
estimate:
\begin{equation}\label{eq:crucial}
  \left\||D|^{\frac12}\int_0^t e^{i(t-s)\Delta}F(s,\cdot)ds\right\|_{L^pL^q}
  \lesssim
  \sum_{j\in\Z}2^{\frac j2}\|F_j\|_{L^2L^2},
\end{equation}
for any admissible couple $(p,q)$ as above, where
\begin{equation*}
  F=\sum_{j\in\Z}F_j,
  \qquad
  \text{supp}F_j\subset\{2^j\leq|x|\leq2^{j+1}\}\times \R
\end{equation*}
Estimate \eqref{eq:crucial} was proved in \cite{RV} first; actually
it follows by mixing the free Strichartz estimates for $T(t)$ with
the dual of the local smoothing estimates which were proved
independently by \cite{cs}, \cite{s} and \cite{v}. In the paper
\cite{RV} the endpoint estimate for $p=2$ is not proved (and indeed it
predates the Keel-Tao paper \cite{KeelTao98-endpoinstric}).
The endpoint case $p=2$ in dimension $n\geq3$ is a consequence of
Lemma 3 in \cite{IK}.

Finally, we need to recall the local smoothing estimates for the
magnetic propagator $e^{itH}$ proved in \cite{FV} under assumptions
less restrictive than the ones of the present paper: we have
\begin{equation}\label{eq:fanellivega}
  \sup_{R>0}\frac1R\int\int_{|x|\leq
  R}|\nabla_Ae^{itH}\varphi|dxdt
  +\sup_{R>0}\frac1{R^2}\int\int_{|x|=R}|e^{itH}\varphi|^2d\sigma_Rdt
  \lesssim
  \|(-\Delta_A)^{\frac14}\varphi\|_{L^2}^2,
\end{equation}
where the constant in the inequality only depends on $B_\tau$ and
$(\partial_rV)_+$.

We are now ready to prove Theorem \ref{thm:main}.
Since
$\text{div}A=0$, we can expand $H$ as follows:
\begin{equation}\label{eq:expansion}
  H=-\Delta+2iA\cdot\nabla_A-|A|^2+V.
\end{equation}
As a consequence, by the Duhamel formula we can write
\begin{equation}\label{eq:duhamel}
  e^{itH}\varphi=e^{it\Delta}\varphi+\int_0^te^{i(t-s)\Delta}R(x,D)e^{itH}\varphi,
  ds,
\end{equation}
where the perturbative operator $R(x,D)$ is given by
\begin{equation}\label{eq:perturb}
  R(x,D)=2iA\cdot\nabla_A-|A|^2+V.
\end{equation}
By \eqref{eq:free1} and \eqref{eq:crucial} we have
\begin{equation}\label{eq:prima}
  \left\||D|^{\frac12}e^{-itH}\varphi\right\|_{L^pL^q}
  \leq C\||D|^{\frac12}\varphi\|_{L^2}+
  \sum_{j\in\Z}2^{\frac
  j2}\left\|\chi_jR(x,D)e^{itH}\varphi\right\|_{L^2L^2},
\end{equation}
where $\chi_j$ is the characteristic function of the ring
$2^j\leq|x|\leq2^{j+1}$. For the first term at the RHS of
\eqref{eq:prima}, by \eqref{eq:magsob2B} we have
\begin{equation}\label{eq:seconda}
  \||D|^{\frac12}\varphi\|_{L^2}\leq C\|H^{\frac14}\varphi\|_{L^2}.
\end{equation}
On the other hand, we can split the second term as follows
\begin{align}\label{eq:terza}
  & \sum_{j\in\Z}2^{\frac j2}\left\|\chi_jR(x,D)e^{itH}\varphi\right\|_{L^2L^2}
  \\
  &
  \leq
  \sum_{j\in\Z}2^{\frac j2}\left(\left\|\chi_jA\cdot\nabla_Ae^{itH}\varphi\right\|_{L^2L^2}
  +
  \left\|\chi_j(V-|A|^2)e^{itH}\varphi\right\|_{L^2L^2}\right)
  =I+II.
  \nonumber
\end{align}
By H\"older inequality, assumption \eqref{ass:decay} and the
smoothing estimates \eqref{eq:fanellivega} we have
\begin{align}\label{eq:quarta}
  I&
  \leq\left(\sum_{j\in\Z}2^{j}\sup_{|x|\sim 2^{j}}|A|\right)\cdot
  \left(\sup_{R>0}\frac1R\int\int_{|x|\leq
  R}|\nabla_Ae^{itH}\varphi|^2dxdt\right)^{\frac12}
  \lesssim
  \|(-\Delta_A)^{\frac14}\varphi\|_{L^2}
  \\
  II & \leq\left(\sum_{j\in\Z}2^{2j}\sup_{|x|\sim
  2^j}(|V|+|A|^2)\right)\cdot
  \left(\sup_{R>0}\frac1{R^2}\int\int_{|x|=
  R}|e^{itH}\varphi|^2d\sigma_Rdt\right)^{\frac12}
  \label{eq:quinta}
  \\ &
  \lesssim
  \|(-\Delta_A)^{\frac14}\varphi\|_{L^2}^2
  .\nonumber
\end{align}
Now we remark that all the assumptions of Theorem \ref{thm:equiv}
are satisfied. Indeed, we know that
$A\in L^2_{\text{loc}}$; moreover, assumption \eqref{ass:decay} 
implies that $|A|\lesssim 1/|x|$ and $|V|\lesssim 1/|x|^{2}$,
hence \eqref{eq:ass2A} is satisfied. 
Thus by Theorem \ref{thm:equiv} (which holds also in the
special case $V\equiv0$) we get
\begin{align*}
\|(-\Delta_A)^\frac 14\varphi\|_{L^2} \leq C \|\varphi\|_{\dot
H^\frac 12}\lesssim\|H^\frac 14 \varphi\|_{L^2}.
\end{align*}
Collecting \eqref{eq:terza}, \eqref{eq:quarta} 
and \eqref{eq:quinta} we obtain
\begin{equation}\label{eq:sesta}
  \sum_{j\in\Z}2^{\frac j2}\left\|\chi_jR(x,D)e^{itH}\varphi\right\|_{L^2L^2}
  \lesssim
  \|H^\frac 14 \varphi\|_{L^2}
\end{equation}
and by \eqref{eq:prima}, \eqref{eq:seconda} and
\eqref{eq:sesta} we deduce
\begin{equation}\label{eq:settima}
  \left\||D|^{\frac12}e^{-itH}\varphi\right\|_{L^pL^q}
  \lesssim
  \|H^\frac 14 \varphi\|_{L^2},
\end{equation}
for any admissible couple $(p,q)$; notice that this includes aldo the 3D
endpoint estimate \eqref{eq:strichartz3D}. In order to conclude the
proof, it is now sufficient to use estimate \eqref{eq:magsob2A} 
which gives
\begin{equation}\label{eq:ottava}
  \left\|H^{\frac14}e^{-itH}\varphi\right\|_{L^pL^q}
  \lesssim
  \|H^\frac 14 \varphi\|_{L^2},
\end{equation}
and commuting $H^{1/4}$ with the flow $e^{itH}$
we obtain \eqref{eq:strichartz}.
However, in dimension 3 \eqref{eq:magsob2A} does
not cover the endpoint $q=6$ and we are left with \eqref{eq:settima}.

\end{document}